\documentclass[a4paper]{amsart}

\usepackage{amsfonts, amsmath, amsthm, amssymb, pifont}
\usepackage{enumitem}
\usepackage[all]{xy}
\usepackage{mathabx}
\usepackage{cancel}

\usepackage{tikz}
\usetikzlibrary{matrix, arrows} 

\usepackage[style=numeric,
            sorting=nyt,
	        backend=bibtex,
			url=false,
			doi=false,
			isbn=false,
			firstinits=true]{biblatex}

\renewbibmacro{in:}{%
	\ifentrytype{article}{}{\printtext{\bibstring{in}\space}}}

\DeclareFieldFormat
[article,inbook,incollection,inproceedings,patent,thesis,unpublished]
{title}{#1\isdot}

\renewbibmacro{volume+number+eid}{%
	\iffieldundef{number}{}{\printtext[parens]{%
			\printfield{number}}}
	\printfield[bold]{volume}
	\setunit{\space}%
	\printfield{eid}}

\renewbibmacro*{series+number}{%
	\printfield{series}%
	\setunit*{\addspace}%
	\printfield{number}%
	\newunit\nopunct}

\DeclareFieldFormat[article]{pages}{#1}

\renewbibmacro*{publisher+location+date}{%
	\printlist{location}%
	\iflistundef{publisher}
	{\setunit*{\addcomma\space}}
	{\setunit*{\addcolon\space}}%
	\printtext[parens]{\printlist{publisher}%
		\setunit*{\addcomma\space}%
		\usebibmacro{date}}
	\newunit}

\renewbibmacro*{location+date}{%
	\printlist{location}%
	\setunit*{\space}%
	\printtext[parens]{\usebibmacro{date}}%
	\newunit}

\NewBibliographyString{submitted}
\DefineBibliographyStrings{english}{%
  submitted = {submitted},
}

\theoremstyle{plain}
\newtheorem{thm}{Theorem}[section]

\newtheorem{lem}[thm]{Lemma}

\newtheorem{cor}[thm]{Corollary}

\theoremstyle{definition}
\newtheorem{defn}[thm]{Definition}

\newtheorem*{alg*}{Algorithm}

\theoremstyle{remark}
\newtheorem{rem}[thm]{Remark}

\newcommand{\iso}{\cong}
\newcommand{\NN}{\mathbb{N}}
\newcommand{\ZZ}{\mathbb{Z}}

\newcommand{\CC}{\mathbb{C}}
\newcommand{\weyl}{\Delta}
\newcommand{\dweyl}{\nabla}

\newcommand{\modcat}[1]{{#1}{\rm -mod}}

\DeclareMathOperator{\soc}{soc}
\DeclareMathOperator{\rad}{rad}
\DeclareMathOperator{\head}{head}
\DeclareMathOperator{\low}{lower}
\DeclareMathOperator{\up}{upper}

\DeclareMathOperator{\midd}{mid}

\DeclareMathOperator{\pr}{pr}

\begin{document}

\title{Balanced semisimple filtrations for tilting modules}
\author{Amit Hazi}
\date{9 October 2015}
\address{Department of Pure Mathematics and Mathematical Statistics\\
Centre for Mathematical Sciences\\
University of Cambridge\\
Wilberforce Road\\
Cambridge\\
CB3 0WB \\
United Kingdom}
\email{A.Hazi@dpmms.cam.ac.uk}

\begin{abstract}
Let $U_l$ be a quantum group at an $l$th root of unity. Many tilting modules for $U_l$ have been shown to have what we call a balanced semisimple filtration, or a Loewy series whose semisimple layers are symmetric about some middle layer. The existence of such filtrations suggests a remarkably straightforward algorithm for calculating these characters if the irreducible characters are already known. We first show that the results of this algorithm agree with Soergel's character formula for the regular tilting modules. We then show that these balanced semisimple filtrations really do exist for these tilting modules.
\end{abstract}

\maketitle

\section*{Introduction}

Let $U_l$ be a quantized universal enveloping algebra at an $l$th root of unity, corresponding to some complex semisimple Lie algebra. The tilting modules $T_l(\lambda)$ of $U_l$ are classified by highest weight $\lambda$ and are characterized by having filtrations by Weyl modules and dual Weyl modules. We are interested in calculating their Loewy series and determining their structure in general.

Andersen and Kaneda showed that $T_l(\lambda)$ is rigid (i.e.~has identical radical and socle series) for $\lambda$ sufficiently high \cite{andersen-kaneda}. In particular, because of self-duality this implies that if the Loewy length of $T_l(\lambda)$ is $2N+1$, we have $\rad_{i+N} T_l(\lambda) \iso \rad_{i-N} T_l(\lambda)$ for any $i$. In other words, the Loewy series is symmetric about the middle layer containing $L_l(\lambda)$. We call such Loewy series \emph{balanced}. Additionally the examples in \cite{andersen-kaneda} and in previous work by Bowman-Doty-Martin \cite{bdm-small,bdm-large} and the author \cite{rigid-tilting} show that the unique Loewy series is compatible with certain Loewy series of the Weyl modules (called dual parity filtrations). This suggests the following algorithm for calculating the character of $T_l(\lambda)$ given the characters of the Weyl modules of weight up to $\lambda$.

\begin{alg*} \hfill
\begin{enumerate}[label=\arabic*.]
\item Write the dual parity filtration of the Weyl module $\weyl_l(\lambda)$. View this as a partial Loewy series for $T_l(\lambda)$ (namely the bottom layers). We will reflect Loewy layers about the ``middle'' Loewy layer in which $L_l(\lambda)$ appears.

\item Pick the highest ``unbalanced'' weight; that is, the largest $\mu<\lambda$ such that $L_l(\mu)$ appears in the series below $L_l(\lambda)$ but there is no corresponding factor $L_l(\mu)$ in the reflected layer above $L_l(\lambda)$. \label{item:restart}

\item Add the dual parity filtration of $\weyl(\mu)$ to the partial Loewy series so that the head of $\weyl(\mu)$ is in the reflected Loewy layer above $L_l(\lambda)$.

\item Repeat from Step \ref{item:restart} until the Loewy series is balanced.
\end{enumerate}
\end{alg*}

For example, the algorithm applied to $T_l(9)$ for the quantum group corresponding to the root system $B_2$ \cite{andersen-kaneda} yields the following partial Loewy series (note that in these pictures we omit the notation $L_l(-)$).

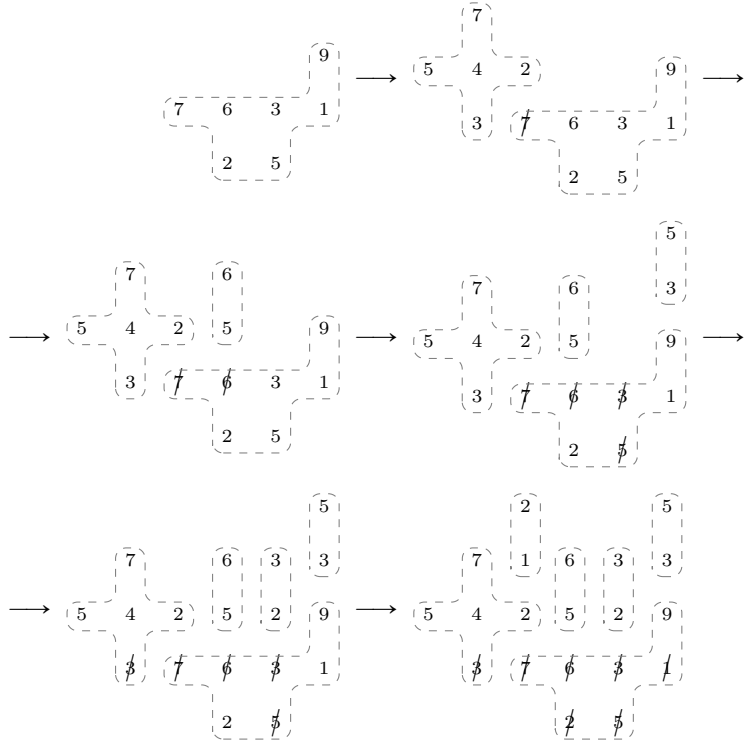
\begin{align*}
\begin{tikzpicture}[>=angle 90, baseline]
\tikzstyle{weylcircles}=[gray, thin, dashed]
\matrix(m)[matrix of math nodes,
row sep=1em, column sep=0.75em,
text height=0.75ex, text depth=0.125ex, ampersand replacement=\&,font=\scriptsize]
{
    \& \&  \& \& \&  \\
    \&  \&  \&  \&  \&  \\
     \&  \&  \&  \&  \& 9 \\
    \&  \& 7 \& 6 \& 3 \& 1 \\
    \& \& \& 2 \& 5 \& \\};
\draw[weylcircles, rounded corners] (m-5-4.south west) -| (m-4-5.south east) -| (m-3-6.north east) -| (m-4-6.north west) -| (m-4-3.south west) -| (m-5-4.south west) -- cycle;
%
\end{tikzpicture} & \longrightarrow
\begin{tikzpicture}[>=angle 90, baseline]
\tikzstyle{weylcircles}=[gray, thin, dashed]
\matrix(m)[matrix of math nodes,
row sep=1em, column sep=0.75em,
text height=0.75ex, text depth=0.125ex, ampersand replacement=\&,font=\scriptsize]
{
    \& \&  \& \& \&  \\
    \& 7 \&  \&  \&  \&  \\
    5 \& 4 \& 2 \&  \&  \& 9 \\
    \& 3 \& \cancel{7} \& 6 \& 3 \& 1 \\
    \& \& \& 2 \& 5 \& \\};
\draw[weylcircles, rounded corners] (m-5-4.south west) -| (m-4-5.south east) -| (m-3-6.north east) -| (m-4-6.north west) -| (m-4-3.south west) -| (m-5-4.south west) -- cycle;
\draw[weylcircles, rounded corners] (m-4-2.south west) -| (m-3-2.south east) -| (m-3-3.north east) -| (m-2-2.north east) -| (m-3-2.north west) -| (m-3-1.south west) -| (m-4-2.south west) -- cycle;
%
\end{tikzpicture} \longrightarrow \\
\longrightarrow \begin{tikzpicture}[>=angle 90, baseline]
\tikzstyle{weylcircles}=[gray, thin, dashed]
\matrix(m)[matrix of math nodes,
row sep=1em, column sep=0.75em,
text height=0.75ex, text depth=0.125ex, ampersand replacement=\&,font=\scriptsize]
{
    \& \&  \& \& \&  \\
    \& 7 \&  \& 6 \&  \&  \\
    5 \& 4 \& 2 \& 5 \&  \& 9 \\
    \& 3 \& \cancel{7} \& \cancel{6} \& 3 \& 1 \\
    \& \& \& 2 \& 5 \& \\};
\draw[weylcircles, rounded corners] (m-5-4.south west) -| (m-4-5.south east) -| (m-3-6.north east) -| (m-4-6.north west) -| (m-4-3.south west) -| (m-5-4.south west) -- cycle;
\draw[weylcircles, rounded corners] (m-4-2.south west) -| (m-3-2.south east) -| (m-3-3.north east) -| (m-2-2.north east) -| (m-3-2.north west) -| (m-3-1.south west) -| (m-4-2.south west) -- cycle;
\draw[weylcircles, rounded corners] (m-3-4.south west) -| (m-2-4.north east) -| (m-3-4.south west) -- cycle;
%
\end{tikzpicture} & \longrightarrow 
\begin{tikzpicture}[>=angle 90, baseline]
\tikzstyle{weylcircles}=[gray, thin, dashed]
\matrix(m)[matrix of math nodes,
row sep=1em, column sep=0.75em,
text height=0.75ex, text depth=0.125ex, ampersand replacement=\&,font=\scriptsize]
{
    \& \&  \& \& \& 5 \\
    \& 7 \&  \& 6 \&  \& 3 \\
    5 \& 4 \& 2 \& 5 \&  \& 9 \\
    \& 3 \& \cancel{7} \& \cancel{6} \& \cancel{3} \& 1 \\
    \& \& \& 2 \& \cancel{5} \& \\};
\draw[weylcircles, rounded corners] (m-5-4.south west) -| (m-4-5.south east) -| (m-3-6.north east) -| (m-4-6.north west) -| (m-4-3.south west) -| (m-5-4.south west) -- cycle;
\draw[weylcircles, rounded corners] (m-4-2.south west) -| (m-3-2.south east) -| (m-3-3.north east) -| (m-2-2.north east) -| (m-3-2.north west) -| (m-3-1.south west) -| (m-4-2.south west) -- cycle;
\draw[weylcircles, rounded corners] (m-3-4.south west) -| (m-2-4.north east) -| (m-3-4.south west) -- cycle;
\draw[weylcircles, rounded corners] (m-2-6.south west) -| (m-1-6.north east) -| (m-2-6.south west) -- cycle;
%
\end{tikzpicture} \longrightarrow \\
\longrightarrow \begin{tikzpicture}[>=angle 90, baseline]
\tikzstyle{weylcircles}=[gray, thin, dashed]
\matrix(m)[matrix of math nodes,
row sep=1em, column sep=0.75em,
text height=0.75ex, text depth=0.125ex, ampersand replacement=\&,font=\scriptsize]
{
    \& \&  \& \& \& 5 \\
    \& 7 \&  \& 6 \& 3 \& 3 \\
    5 \& 4 \& 2 \& 5 \& 2 \& 9 \\
    \& \cancel{3} \& \cancel{7} \& \cancel{6} \& \cancel{3} \& 1 \\
    \& \& \& 2 \& \cancel{5} \& \\};
\draw[weylcircles, rounded corners] (m-5-4.south west) -| (m-4-5.south east) -| (m-3-6.north east) -| (m-4-6.north west) -| (m-4-3.south west) -| (m-5-4.south west) -- cycle;
\draw[weylcircles, rounded corners] (m-4-2.south west) -| (m-3-2.south east) -| (m-3-3.north east) -| (m-2-2.north east) -| (m-3-2.north west) -| (m-3-1.south west) -| (m-4-2.south west) -- cycle;
\draw[weylcircles, rounded corners] (m-3-4.south west) -| (m-2-4.north east) -| (m-3-4.south west) -- cycle;
\draw[weylcircles, rounded corners] (m-2-6.south west) -| (m-1-6.north east) -| (m-2-6.south west) -- cycle;
\draw[weylcircles, rounded corners] (m-3-5.south west) -| (m-2-5.north east) -| (m-3-5.south west) -- cycle;
%
\end{tikzpicture} & \longrightarrow \begin{tikzpicture}[>=angle 90, baseline]
\tikzstyle{weylcircles}=[gray, thin, dashed]
\matrix(m)[matrix of math nodes,
row sep=1em, column sep=0.75em,
text height=0.75ex, text depth=0.125ex, ampersand replacement=\&,font=\scriptsize]
{
    \& \& 2 \& \& \& 5 \\
    \& 7 \& 1 \& 6 \& 3 \& 3 \\
    5 \& 4 \& 2 \& 5 \& 2 \& 9 \\
    \& \cancel{3} \& \cancel{7} \& \cancel{6} \& \cancel{3} \& \cancel{1} \\
    \& \& \& \cancel{2} \& \cancel{5} \& \\};
\draw[weylcircles, rounded corners] (m-5-4.south west) -| (m-4-5.south east) -| (m-3-6.north east) -| (m-4-6.north west) -| (m-4-3.south west) -| (m-5-4.south west) -- cycle;
\draw[weylcircles, rounded corners] (m-4-2.south west) -| (m-3-2.south east) -| (m-3-3.north east) -| (m-2-2.north east) -| (m-3-2.north west) -| (m-3-1.south west) -| (m-4-2.south west) -- cycle;
\draw[weylcircles, rounded corners] (m-3-4.south west) -| (m-2-4.north east) -| (m-3-4.south west) -- cycle;
\draw[weylcircles, rounded corners] (m-2-6.south west) -| (m-1-6.north east) -| (m-2-6.south west) -- cycle;
\draw[weylcircles, rounded corners] (m-3-5.south west) -| (m-2-5.north east) -| (m-3-5.south west) -- cycle;
\draw[weylcircles, rounded corners] (m-2-3.south west) -| (m-1-3.north east) -| (m-2-3.south west) -- cycle;
\end{tikzpicture}
\end{align*}

In this paper we first prove that this na\"{i}ve algorithm in fact works for all regular tilting modules (not just rigid ones) at the level of characters. The key ingredients in this proof are Lusztig's character formula, which is true for almost all $l$ in the case of quantum groups \cite{KL1-2,KL3,KL4,lusztig-nonsym} and Soergel's tilting character formula \cite{soergel-KL,soergel-kacmoody}. In the final section we prove that the balanced semisimple filtrations alluded to above really do exist for all regular tilting modules. In the future we hope to use similar methods to find a general character formula for tilting modules in the modular case.

\section{Quantum groups at roots of unity}

Let $R$ be a root system for a Euclidean space $E$ of dimension $n$, and let $A_R$ be the Cartan matrix associated to this root system. Let $q$ be an indeterminate in the ring $\mathcal{A}=\ZZ[q^{\pm 1}]$. Write $U_\mathcal{A}$ for the Lusztig integral form quantum group associated to the Cartan matrix $A_R$. This quantum group is a Hopf algebra over $\mathcal{A}$ with algebra generators $E_i^{(r)},F_i^{(r)},K_i^{\pm 1}$ ranging over $i=1,\dotsc,n$ and $r \in \NN$.

Now let $l \in \NN$ be an odd positive integer (with $l \neq 3$ if $R$ has a $G_2$-component). Set $\zeta=e^{2\pi i/l} \in \CC$, a primitive $l$th root of unity. We can make $\mathcal{A}$ into a commutative $\CC$-algebra by specializing $q$ to $\zeta$. This leads to a specialization $U_l=\CC \otimes_\mathcal{A} U_\mathcal{A}$ of our quantum group at $\zeta$.

We will restrict ourselves to the study of finite-dimensional $U_l$-modules of type $\mathbf{1}$ (see \cite[H.10]{jantzen} for a precise definition). The representation theory of $U_l$-modules is analogous to the representation theory of an algebraic group $G$ with root system $R$ over a field of characteristic $l$. In particular, if $R^+$ denotes the set of positive roots, and we define
\begin{align*}
X& =\{\lambda \in E: \langle\lambda,\alpha^\vee \rangle \in \ZZ \text{ for all } \alpha \in R^+\} \\
X^+& =\{\lambda \in E: \langle\lambda,\alpha^\vee \rangle \in \ZZ_{\geq 0} \text{ for all } \alpha \in R^+\}
\end{align*}
to be the sets of integral and dominant integral weights respectively, then for each $\lambda \in X^+$ we have $U_l$-modules $\dweyl_l(\lambda)$ and $\weyl_l(\lambda)$. The module $L_l(\lambda)=\soc \dweyl_l(\lambda) \iso \weyl_l(\lambda)/\rad \weyl_l(\lambda)$ is simple, and all simple modules are of this form. Moreover, familiar results from the theory of algebraic groups (including Kempf's vanishing theorem) carry over for these $U_l$-modules. This means that we can define the indecomposable tilting module $T_l(\lambda)$, in a manner completely analogous to the $G$-modules case, as the unique indecomposable module with a $\weyl$-filtration and a $\dweyl$-filtration with highest weight $\lambda$.

The affine Weyl group $\mathcal{W}$ is defined to be $l\ZZ R \rtimes W$, i.e.~the group of isometries of $E$ generated by translations by the scaled root lattice $l\ZZ R$ and the Weyl group $W$ of $R$. It normally acts on weights via the dot action, which shifts the origin to $-\rho$:
\begin{equation*}
w \cdot \lambda=w(\lambda+\rho)-\rho
\end{equation*}
The weight space $X$ can be divided into simplicial fundamental regions called alcoves. The bottom alcove $C_\mathrm{bot}$ and its closure $\overline{C}_\mathrm{bot}$ are defined to be the sets of integral weights
\begin{align*}
C_\mathrm{bot}& =\{\lambda \in X: 0<\langle \lambda+\rho,\alpha^\vee \rangle<l \text{ for all } \alpha \in R^+\} \\
\overline{C}_\mathrm{bot}& =\{\lambda \in X: 0\leq \langle \lambda+\rho,\alpha^\vee \rangle\leq l \text{ for all } \alpha \in R^+\}
\end{align*}
The other alcoves can be obtained by taking the image of $C_\mathrm{bot}$ under some isometry from the affine Weyl group $\mathcal{W}$. The dominant alcoves are those which intersect the dominant region $X^+$ non-trivially. 

The linkage principle states that if $L_l(\lambda)$ and $L_l(\lambda')$ are in the same block, then $\lambda' \in \mathcal{W} \cdot \lambda$. We write $\mathcal{B}_\lambda$ for the full subcategory of modules whose composition factors all lie in $\mathcal{W} \cdot \lambda$, and $\pr_\lambda:\modcat{U_l} \rightarrow \mathcal{B}_\lambda$ for the projection functor onto this subcategory. For a dominant alcove $C$ and $\lambda,\mu \in \overline{C}$ the translation functor is defined by
\begin{equation*}
T_\lambda^\mu(V)=\pr_\mu\left(\pr_\lambda(V) \otimes L_l(w(\mu-\lambda))\right)
\end{equation*}
where $w \in W$ is chosen so that $w(\mu-\lambda) \in X^+$. Note that $T_\lambda^\mu$ is always exact as the composition of several exact functors. The translation principle states that $T_\lambda^\mu,T_\mu^\lambda:\mathcal{B}_\lambda \leftrightarrows \mathcal{B}_\mu$ are adjoint and mutually inverse if $\lambda$ and $\mu$ belong to the same set of alcoves.

Suppose $\lambda,\lambda' \in X^+$ belong to adjacent alcoves $C,C'$ with $\lambda<\lambda'$. Let $\mu$ be a weight on the wall between them, labelled by $s \in W$. The wall-crossing functor is defined to be $\theta_s=T_\mu^{\lambda'} \circ T_\lambda^\mu$, which is self-adjoint and exact. It is well-known that $\theta_s \weyl(\lambda)\iso \theta_s \weyl(\lambda')$, and we have the exact sequence
\begin{equation}
0 \rightarrow \weyl_l(\lambda') \rightarrow \theta_s \weyl_l(\lambda) \rightarrow \weyl_l(\lambda) \rightarrow 0
\end{equation}

\section{Kazhdan-Lusztig combinatorics}

\subsection*{Notation}

We use notation from \cite{soergel-KL} and \cite[Appendix C.1]{jantzen} for various Kazhdan-Lusztig polynomials, which we will summarize.

The affine Weyl group $\mathcal{W}$ can be viewed as a Coxeter group, with $(n+1)$ generators $\mathcal{S}$ corresponding to reflections in the walls of the bottom alcove $C$. It comes equipped with the Bruhat order and the length function $\ell:\mathcal{W} \rightarrow \ZZ_{\geq 0}$. Let $S$ denote the $n$ generators in $\mathcal{S}$ corresponding to reflections contained in the Weyl group $W$. 

We write $\mathcal{L}=\ZZ[v^{\pm 1}]$. Let $\mathcal{H}=\mathcal{H}(\mathcal{W},\mathcal{S})$ be the Hecke algebra associated to the Coxeter system $(\mathcal{W},\mathcal{S})$, an associative $\mathcal{L}$-algebra with generators $\{H_s\}_{s \in \mathcal{S}}$ and relations
\begin{align}
H_s^2 & =1+(v^{-1}-v)H_s && \text{for all $s \in \mathcal{S}$} \\
\overbrace{H_s H_t H_s \dotsm}^{\text{$r$ terms}} & =\overbrace{H_t H_s H_t \dotsm}^{\text{$r$ terms}}&& \text{for all $s,t \in \mathcal{S}$, where $r$ is the order of $st$}
\end{align}
For any reduced word $x=stu\dotsm \in \mathcal{W}$, the element $H_x=H_s H_t H_u \dotsm$ is well-defined, and the set $\{H_x\}_{x \in \mathcal{W}}$ forms an $\mathcal{L}$-basis for $\mathcal{H}$. Each generator $H_s$ is invertible, with $H_s^{-1}=H_s+v-v^{-1}$, so each basis element $H_x$ is also invertible. Define the ring homomorphism $d:\mathcal{H} \rightarrow \mathcal{H}$ by $d(v)=v^{-1}$ and $d(H_x)=(H_{x^{-1}})^{-1}$ (this extends an obvious homomorphism on $\mathcal{L}$). We call this involution dualization, and we sometimes write $\overline{H}$ for $d(H)$. For $s \in \mathcal{S}$ we define $\underline{H}_s=H_s+v$ and $\tilde{\underline{H}}_s=H_s-v^{-1}$. Notice that $\overline{H}_s=H_s+v-v^{-1}$ so both $\underline{H}_s$ and $\tilde{\underline{H}}_s$ are self-dual, and the sets $\{\underline{H}_s\}_{s \in \mathcal{S}}$ and $\{\tilde{\underline{H}}_s\}_{s \in \mathcal{S}}$ each generate $\mathcal{H}$ as an $\mathcal{L}$-algebra.

Now let $\mathcal{H}_W=\mathcal{H}(W,S) \leq \mathcal{H}$ be the Hecke algebra obtained from the Weyl group. Since $(H_s-v^{-1})(H_s+v)=0$ for each generator $s \in S$, for each $u \in \{-v,v^{-1}\}$ there is a homomorphism of $\mathcal{L}$-algebras $\varphi_u: \mathcal{H}_W \rightarrow \mathcal{L}$, defined by mapping $H_s \mapsto u$. This turns $\mathcal{L}$ into a right $\mathcal{H}_W$-module which we call $\mathcal{L}(u)$. Now define two right $\mathcal{H}$-modules
\begin{align*}
\mathcal{M}& =\mathcal{L}(v^{-1}) \otimes_{\mathcal{H}_W} \mathcal{H} \\
\mathcal{N}& =\mathcal{L}(-v) \otimes_{\mathcal{H}_W} \mathcal{H}
\end{align*}

We can obtain an $\mathcal{L}$-basis for $\mathcal{M}$ via a set of representatives for the right cosets $W \backslash \mathcal{W}$. A natural choice for such representatives comes from the dominant alcoves, namely, the set $\mathcal{W}^+=\{x \in \mathcal{W} : (x \cdot C) \cap X^+ \neq \emptyset\}$, or in other words the affine Weyl group elements which map $C$ to another dominant alcove. Defining $M_x$ to be $1 \otimes H_x$ in $\mathcal{M}$, we get the $\mathcal{L}$-basis $\{M_x\}_{x \in \mathcal{W}^+}$. The entirely analogous idea works for $\mathcal{N}$. The action of $\underline{H}_s$ on these bases is
\begin{align}
M_x\underline{H}_s& =\begin{cases}
M_{xs}+vM_x & \text{if $xs \in \mathcal{W}^+$ and $xs>x$} \\
M_{xs}+v^{-1}M_x & \text{if $xs \in \mathcal{W}^+$ and $xs<x$} \\
(v+v^{-1})M_x & \text{if $xs \notin \mathcal{W}^+$}
\end{cases} \\ 
N_x\underline{H}_s& =\begin{cases}
N_{xs}+vN_x & \text{if $xs \in \mathcal{W}^+$ and $xs>x$} \\
N_{xs}+v^{-1}N_x & \text{if $xs \in \mathcal{W}^+$ and $xs<x$} \\
0 & \text{if $xs \notin \mathcal{W}^+$}
\end{cases}
\end{align}

The dualization map $d:\mathcal{H} \rightarrow \mathcal{H}$ extends to a dualization map of $\mathcal{M}$ by mapping $a \otimes H \mapsto \overline{a} \otimes \overline{H}$. To see this, note that for all $s \in \mathcal{S}$
\begin{equation*}
\phi_u(\underline{H}_s)=\begin{cases}
v+v^{-1} & \text{if $u=v^{-1}$} \\
0 & \text{if $u=-v$}
\end{cases}
\end{equation*}
so $\phi_u(\underline{H}_s)$ is self-dual. This means that for $s \in S$,
\begin{align*}
d(a \otimes (\underline{H}_s H))& =\overline{a} \otimes \overline{\underline{H}_s} \overline{H} \\
& =\overline{a} \otimes \underline{H}_s \overline{H} \\
& =\overline{a}\phi_u(\underline{H}_s) \otimes \overline{H} \\
& =\overline{a\phi_u(\underline{H}_s)} \otimes \overline{H} \\
& =d(a\phi_u(\underline{H}_s) \otimes H)
\end{align*}
As $\{\underline{H}_s\}_{s \in S}$ generates $\mathcal{H}_W$ this shows that the map above is well-defined.

\begin{thm}[{\cite[Theorem 3.1]{soergel-KL}}]
There is a unique set of polynomials $\{m_{y,x}\}_{x,y \in \mathcal{W}^+}$ in $\ZZ[v]$ such that
\begin{enumerate}[label={\upshape (\roman*)}]
\item if $m_{y,x} \neq 0$, then either $y=x$ and $m_{x,x}=1$ or $y<x$ and $m_{y,x} \in v\ZZ[v]$;

\item the element $\underline{M}_x=\sum_{y} m_{y,x} M_y$ is self-dual.
\end{enumerate}
There are also analogous polynomials $\{n_{y,x}\}_{x,y \in \mathcal{W}^+}$ for $\mathcal{N}$.
\end{thm}
\begin{proof}
Induct on the length of $x$. Suppose for some $x \in \mathcal{W}^+$ we have already defined $\underline{M}_x$ and all $\underline{M}_u$ with $\ell(u)<\ell(x)$. Suppose $s \in \mathcal{S}$ such that $xs \in \mathcal{W}^+$ and $xs>x$. Write 
\begin{equation*}
\underline{M}_x \underline{H}_s=M_{xs}+\sum_{y<x} m^s_{y,x} M_y
\end{equation*}
From the action of $\underline{H}_s$ on the basis above we have (for $x,y \in \mathcal{W}^+$)
\begin{equation}
m^s_{y,x}=\begin{cases}
m_{ys,x}+v^{-1}m_{y,x} & \text{if $ys<y$ and $ys \in \mathcal{W}^+$} \\
m_{ys,x}+vm_{y,x} & \text{if $y<ys$ and $ys \in \mathcal{W}^+$} \\
(v+v^{-1})m_{y,x} & \text{if $ys \notin \mathcal{W}^+$}
\end{cases}
\end{equation}
Clearly $\underline{M}_x \underline{H}_s$ is self-dual, so the element
\begin{equation*}
\underline{M}_{xs}=\underline{M}_x \underline{H}_s-\sum_{y<x} m^s_{y,x}(0) \underline{M}_y=M_{xs}+\sum_{y<x} m_{y,xs} M_y
\end{equation*}
is also self-dual, with the property that the coefficient $m_{y,xs}$ of $M_y$ has zero constant coefficient.
\end{proof}

\begin{rem}
For self-dual elements of the form 
\begin{equation*}
\tilde{\underline{M}}_x=M_x+\sum_{y<x} \tilde{m}_{y,x} M_y
\end{equation*}
where the polynomial coefficients are restricted to being in $v^{-1}\ZZ[v^{-1}]$ instead of $v\ZZ[v]$, Deodhar showed that we must have $\tilde{m}_{y,x}=(-1)^{\ell(x)+\ell(y)} \overline{n_{y,x}}$, and similarly for $\tilde{\underline{N}}_x$ \cite{deodhar}. It is not difficult to show that an inductive construction similar to the above proof uses multiplication by $\tilde{\underline{H}}_s$ instead of $\underline{H}_s$.
\end{rem}

We can now define the inverse polynomials $\{m^{y,x}\}$ for $y,x \in \mathcal{W}^+$ and $y \geq x$ such that the following formula holds:
\begin{equation}
\sum_z (-1)^{\ell(z)+\ell(x)} m^{z,x} m_{z,y}=\delta_{x,y}
\end{equation}
While there is an interpretation of these polynomials as coefficients of some element of $\mathcal{M}^\ast$ with respect to a certain basis, this will not matter for the sequel.

\subsection*{Character formulae}

Let $A$ be some dominant alcove. The structure of the module $\weyl_l(\lambda)$ for any $\lambda$ contained in the interior of $A$ only depends on $A$ and not on the exact weight $\lambda$ by the translation principle. So we may abuse notation and write $\weyl_l(A)$ instead of $\weyl_l(\lambda)$. We can even reconstruct character formulae written in this way using the linkage principle. We will also freely use the bijection between dominant alcoves and elements of $\mathcal{W}^+$ for the indices of the various Kazhdan-Lusztig polynomials $\{m_{y,x}\}$ etc.

With this notation, Lusztig's character formula can be written as follows.

\begin{thm}[Lusztig's character formula]
Let $A$ be an alcove in the dominant region. Then the following character formula holds:
\begin{equation}
[\weyl_l(A)]=\sum_B m^{A,B}(1) [L_l(B)]
\end{equation}
where the sum is over all dominant alcoves $B$.
\end{thm}

It is equivalent to the following corollary, which is sometimes called the Vogan conjecture.

\begin{cor}
Let $A$ be an alcove in the dominant region, and $s \in \mathcal{S}$ a simple reflection such that $s \cdot A>A$. Then $\theta_s(L_l(A))$ has socle and head isomorphic to $L_l(A)$, and the module
\begin{equation*}
\beta_s(L_l(A))=\rad \theta_s(L_l(A))/\soc \theta_s(L_l(A))
\end{equation*}
is semisimple.
\end{cor}

If the corollary holds one can show that $[\beta_s(L_l(A)):L_l(B)]=m^s_{A,B}(1)$.

For tilting modules, Soergel stated proved the following character formula \cite{soergel-KL,soergel-kacmoody}.

\begin{thm}[Soergel's tilting character formula]
Let $A$ be an alcove in the dominant region. Then the following character formula holds:
\begin{equation}
[T_l(A)]=\sum_B n_{B,A}(1) [\weyl_l(B)]
\end{equation}
where the sum is over all dominant alcoves $B$.
\end{thm}

The following combinatorial property of the Kazhdan-Lusztig polynomials is the basis for the balancing property of filtrations in later sections.
\begin{lem} \label{lem:kl-balanced}
The polynomial
\begin{equation}
t_{B,A}=\sum_C n_{C,A}\overline{m^{C,B}} \label{eq:KL-tilt-balance}
\end{equation}
is self-dual.
\end{lem}

\begin{proof}
Let $T_A=\sum_B \overline{m^{A,B}} \tilde{\underline{N}}_B$. Unlike $\tilde{\underline{N}}_A$, $T_A$ is not self-dual. Now define $\underline{T}_A$ as follows:
\begin{equation*}
\underline{T}_A=\sum_C n_{C,A} T_K
\end{equation*}
We claim that this sum is self-dual. Since the coefficient of $\tilde{\underline{N}}_B$ in $\underline{T}_A$ is $t_{B,A}$ and $\tilde{\underline{N}}_B$ is self-dual, this shows what we want.
\begin{align*}
\underline{T}_A& =\sum_{B,C} n_{C,A} \overline{m^{C,B}} \tilde{\underline{N}}_B \\
& =\sum_{B,C,D} (-1)^{\ell(B)+\ell(D)} n_{C,A} \overline{m^{C,B}} \overline{m_{D,B}} N_D \\
& =\sum_{C,D} (-1)^{\ell(C)+\ell(D)} n_{C,A} \delta_{C,D} N_D \\
& =\sum_C n_{C,A} N_C \\
& =\underline{N}_A
\end{align*}
\end{proof}

There is an abelian group isomorphism from the Grothendieck group to $\mathcal{N}$ mapping $\underline{N}_A \mapsto [T_l(A)]$. It maps $N_A \mapsto [\weyl_l(A)]$ and $\tilde{\underline{N}}_A \mapsto [L_l(A)]$. Since the action of $\underline{H}_s$ on the basis $\{N_A\}$ matches the action of the wall-crossing functor $\theta_s$ on the Weyl modules (on the level of characters), we have that this holds for any $\weyl$-filtered module, so we can evaluate the character $\theta_s([T_l(A)])$ in this way:
\begin{align*}
\theta_s([T_l(A)])& =[T_l(A)]\underline{H}_s \\
& =\sum_{B,C} n_{C,A} \overline{m^{C,B}} \tilde{\underline{N}}_B\underline{H}_s \\
& =\sum_{B,C} n_{C,A} \overline{m^{C,B}} \tilde{\underline{N}}_B(\tilde{\underline{H}}_s+v+v^{-1}) \\
& =\sum_{B,C} n_{C,A} \overline{m^{C,B}} \tilde{\underline{N}}_B(\tilde{\underline{H}}_s+v+v^{-1}) \\
& =\sum_{B,C} n_{C,A} \overline{m^{C,B}} \left((v+v^{-1})\tilde{\underline{N}}_B+\tilde{\underline{N}}_{s \cdot B}+\sum_D \overline{m^s_{D,B}}\tilde{\underline{N}}_D\right)
\end{align*}
This shows that the polynomial version of $\theta_s([T_l(A)])$ respects the filtration described by Vogan's conjecture, which we will use later.

\section{Balanced semisimple filtrations}

\subsection*{Isotropic filtrations}

Let $V$ be a self-dual $U_l$-module. Fix an isomorphism $\phi:V \rightarrow {}^\tau V$. This isomorphism is equivalent to a non-degenerate bilinear form $(-,-)$ on $V$, with the property that $(xv,v')=(v,\tau(x)v')$ for all $x \in U_l$ and $v,v' \in V$. Such forms are called contravariant \cite[Section II.8.17]{jantzen}. By considering the action of such a form on weight spaces it can be shown that all such forms are symmetric. For a subspace $U$ of $V$, recall that the orthogonal subspace is defined to be $U^\perp=\{v \in V : (u,v)=0 \text{ for all } u \in U \}$.

\begin{defn} 
Suppose $U$ is a submodule of $V$. The submodule $U$ is called totally isotropic if $U \leq U^\perp$. Dually $U$ is called totally coisotropic if we have $U^\perp \leq U$.
\end{defn}

It is immediately clear that $U$ is totally isotropic if and only if $U^\perp$ is totally coisotropic.

\begin{lem}
If $U$ is a totally isotropic submodule of $V$, then $T_\lambda^\mu(U)$ is a totally isotropic submodule of $T_\lambda^\mu(V)$.
\end{lem}
\begin{proof}
Apply $T_\lambda^\mu$ to the following commutative square
\begin{equation*}
\xymatrix{
U \ar[r] \ar[d] & V \ar[d] \\
U^\perp \ar[r] & {}^\tau V
}
\end{equation*}
Since $U^\perp \iso {}^\tau(V/U)$, we have $T_\lambda^\mu(U^\perp) \iso (T_\lambda^\mu U)^\perp$. This means that $T_\lambda^\mu(U)$ maps into $(T_\lambda^\mu U)^\perp$, which implies that $T_\lambda^\mu(U)$ is a totally isotropic submodule of $T_\lambda^\mu(V)$.
\end{proof}

\begin{defn}
A filtration $\{V_i\}$ of $V$ is called isotropic if it can be written in the form
\begin{equation*}
0=V_{m}^\perp \leq \dotsb \leq V_1^\perp \leq \dotsb \leq V_1 \leq \dotsb \leq V_m=V
\end{equation*}
for some $m \geq 0$. In this situation we typically reindex so that $V_{-i}=V_i^\perp$ for $i>0$. We call $V_{-1}$ and $V_1$ the lower half and upper half of $\{V_i\}$ respectively, denoted $\low \{V_i\}$ and $\up \{V_i\}$. We call $\{V_i\}$ maximal isotropic if $\low \{V_i\}$ is maximal, i.e.~if there is no other isotropic filtration $\{V'_{i'}\}$ such that $\low \{V'_{i'}\} \geq \low \{V_i\}$. The subquotient $\up \{V_i\}/\low \{V_i\}$ is called the middle and is denoted $\midd \{V_i\}$.
\end{defn}

We denote the layers of an isotropic filtration by
\begin{equation*}
V^i=\begin{cases}
V_{i+1}/V_i & \text{if $i>0$} \\
V_i/V_{i-1} & \text{if $i<0$} \\
V_1/V_{-1} & \text{if $i=0$}
\end{cases}
\end{equation*}

If $\{V_i\}$ is a maximal isotropic filtration, then $\midd \{V_i\}$ must be semisimple. To see this, suppose otherwise. We have $(\soc \midd \{V_i\})^\perp=\rad \midd \{V_i\}$. For any non-semisimple summand $U$ of $\midd\{V_i\}$ we have $\rad U \geq \soc U$. From this summand we could construct a larger isotropic filtration, which is a contradiction. 

From \cite{andersen-kaneda} and \cite{ak-loewy}, it follows that the dual Weyl modules have parity filtrations determined by the Kazhdan-Lusztig polynomials $m^{B,A}$. Specifically, there exists a filtration $\dweyl_l(A)^i$ of $\dweyl_l(A)$ such that the successive subquotients $\dweyl_l(A)^i=\dweyl_l(A)_{i+1}/\dweyl_l(A)_i$ are all semisimple, with character
\begin{equation*}
[\dweyl_l(A)^i]=\sum_B (m^{A,B})_i [L_l(B)]
\end{equation*}
where $(m^{A,B})_i$ denotes the coefficient of $v^i$ of the polynomial $m^{A,B}$.

\begin{defn}
Suppose $T$ is a tilting module. A semisimple isotropic filtration $\{T_i\}$ of $T$ is called a balanced semisimple filtration if there is a $\weyl$-filtration
\begin{equation*}
0 \leq T_{(\lambda_1)} \leq T_{(\lambda_2)} \leq \dotsb
\end{equation*}
with weights $\lambda_1 \geq \lambda_2 \geq \dotsb$ and $T_{(\lambda_k)}/T_{(\lambda_{k-1})} \iso \weyl(\lambda_k)$ such that the induced filtration on the subquotients is a shifted version of the parity filtration above. Recall that the $i$th part of the induced filtration on the subquotient $T_{(\lambda_k)}/T_{(\lambda_{k-1})}$ is $(T_{(\lambda_k)} \cap T_i +T_{(\lambda_{k-1})})/T_{(\lambda_{k-1})}$.
\end{defn}

\subsection*{Proof of the main theorem}

Suppose we have a composition factor $U/V$ and another composition factor $V/W$. We say that $U/V$ is above $V/W$ if the extension $U/W$ doesn't split. Otherwise there is a submodule $M \leq U$ with $M+V=U$ and $M \cap V=W$. Then we have $U/V=(M+V)/V \iso M/W$ and also $M \leq U$, so $U/M=(M+V)/M \iso V/W$, and we can switch the order of the composition factors.

\begin{thm}
Let $T=T_l(A)$. There exists a balanced semisimple filtration $\{T_i\}$ of $T$ such that
\begin{equation*}
[T^i:L(B)]=(t_{B,A})_i
\end{equation*}
where $(t_{B,A})_i$ denotes the coefficient of $v^i$ in $t_{B,A}$.
\end{thm}

\begin{proof}
Induct on the alcove $A$. The base case is when $A$ is the bottom alcove and we have $T_l(A) \iso L_l(A)$. So there is an isotropic filtration $0=T_0 \leq T_0^\perp=T_l(A)$, which is what we want.

For the inductive step, suppose we have shown that the claim holds for all alcoves below $A$. Pick a simple reflection $s \in \mathcal{S}$ such that $s \cdot A<A$ in the dominance ordering. Then $\theta_s(T_l(s \cdot A)) \iso T_l(A) \oplus Q$ where $Q$ is a tilting module of weight lower than $A$. We will denote $T_l(s\cdot A)$ by $T'$ for simplicity.

By induction there is a balanced semisimple filtration $\{T'_i\}$ of $T'$ satisfying the claim. By a previous lemma, $\{\theta_s(T'_i)\}$ is an isotropic filtration of $T \oplus Q$. Pick complements $T_i$ of $ \cap \theta_s(T'_i)$ in $T$ such that $T_i^\perp=T_{-i}$ for all $i$. This is possible because no summand of $Q$ is isomorphic to $T$. This gives an isotropic filtration of $T$.

Suppose the bottom layer of $T'$ is $T'_{-m}=0$ for some $m\geq 0$. Consider the submodules $0=T_{-m} \leq T_{-m+1} \leq T_{-m+2}$. These describe a filtration for a summand of the module $\theta_s(T'_{-m+2})$. Clearly $T'_{-m+2}$ has Loewy length at most $2$, so $\theta_s(T')$ has a Loewy length of at most $2+2=4$. Now define $T_{-m+4/3}$ such that $T_{-m+4/3}/T_{-m+1} \iso \soc(T_{-m+2}/T_{-m+1})$ and $T_{-m+2/3}$ such that $T_{-m+2/3}/T_{-m} \iso \rad(T_{-m+1}/T_{-m})$. As $T_{-m+4/3}/T_{-m+1}$ is semisimple, any composition factor can be written as $U/T_{-m+1}$, and similarly any composition factor of $T_{-m+1}/T_{-m+2/3}$ can be written $T_{-m+1}/W$. If there is a composition factor $U/T_{-m+1}$ which lies above $T_{-m+1}/W$, then the Loewy length of $T_{-m+2}$ is at least $6$, which is a contradiction. Thus there must be a module $Y$ such that $Y+T_{-m+1}=T_{-m+4/3}$ and $Y \cap T_{-m+1}=T_{-m+2/3}$, so that we can switch the order of these composition factors. 

This leaves us with a semisimple filtration $0=T_{-m} \leq T_{-m+1/3} \leq T_{-m+2/3} \leq Y \leq T_{-m+4/3} \leq T_{-m+5/3} \leq T_{-m+2}$, where we have continued the ``thirds'' notation suggested above. But from the results at the end of the previous section, we see that $Y/T_{-m+2/3} \iso T_{-m+4/3}/T_{-m+1}$ and $T_{-m+2/3}/T_{-m+1/3}$ have the same Kazhdan-Lusztig parity, so in fact $Y/T_{-m+1/3}$ is semisimple. Similarly $T_{-m+5/3}/Y$ is semisimple. With this in mind, we redefine the filtration $\{T_i\}$ so that its first few lower layers are $0 \leq T_{-m+1/3} \leq Y \leq T_{-m+5/3} \leq T_{-m+2}$. We continue in this manner up through the lower half of $T$, re-indexing so that the filtration layers correspond to the polynomials $t_{B,A}$. Obviously by taking orthogonal spaces this works for the upper half as well.

By induction $\midd \{T'_i\}$ is semisimple. Thus $\midd \{T_i\}$ is a summand of $\theta_s(\midd \{T'_i\})$, which is a self-dual module of Loewy length $3$ by Vogan's conjecture. If we define $V$ such that $V/T_{-1} \iso \rad(T_1/T_{-1})$ then we have $T_1/V \iso \head(T_1/T_{-1})$ so $V^\perp/T_{-1} \iso \soc(T_1/T_{-1})$, and $V \geq V^\perp$. Thus $V^\perp$ is a larger totally isotropic submodule of $T$, so we can redefine $T_1$ and $T_{-1}$ to be $V$ and $V^\perp$ respectively. The resulting filtration after all these changes has layers corresponding to the polynomials $t_{B,A}$.

In fact this new filtration is balanced semisimple. The module $T$ naturally has a Weyl filtration because $T'$ does, which we label $T_{(\lambda_k)}$. We claim that $T_{(\lambda_k)} \cap T_i$ has the following character based on a ``partial'' version of $t_{B,A}$:
\begin{equation*}
[T_{(\lambda_k)} \cap T_i:L_l(\lambda)]=\left(\sum_{j \leq k} n_{\lambda_j,\lambda_k} \overline{m^{\lambda_j,\lambda}}\right)_{\leq i}
\end{equation*}
where $(-)_{\leq i}$ denotes the sum of the coefficients of $v^j$ for all $j \leq i$. To see this, note that a similar result holds for the original filtration on $T$, since it was a translated version of a balanced semisimple filtration on $T'$. The modifications made to this filtration don't change the fact that composition factors in the layers $T^i$ can be identified as belonging to some Weyl subquotient.

The induced filtration on $T_{(\lambda_k)}/T^{(\lambda_{k-1})}$ has $i$th layer
\begin{align*}
\frac{(T_i\cap T_{(\lambda_k)}+T_{(\lambda_{k-1})})/T_{(\lambda_{k-1})}}{(T_{i-1}\cap T_{(\lambda_k)}+T_{(\lambda_{k-1})})/T_{(\lambda_{k-1})}}& \iso \frac{T_i\cap T_{(\lambda_k)}+T_{(\lambda_{k-1})}}{T_{i-1}\cap T_{(\lambda_k)}+T_{(\lambda_{k-1})}} \\
& \iso \frac{T_i \cap T_{(\lambda_k)}}{(T_i \cap T_{(\lambda_k)}) \cap (T_{i-1}\cap T_{(\lambda_k)}+T_{(\lambda_{k-1})})} \\
& =\frac{T_i \cap T_{(\lambda_k)}}{T_i \cap T_{(\lambda_k)} \cap T_{(\lambda_{k-1})}+T_{i-1}\cap T_{(\lambda_k)}} \\
& =\frac{T_i \cap T_{(\lambda_k)}}{T_i \cap T_{(\lambda_{k-1})}+T_{i-1}\cap T_{(\lambda_k)}}
\end{align*}
From the above, the character of this $i$th layer is
\begin{equation*}
\left(\sum_{j \leq k} n_{\lambda_j,\lambda_k} \overline{m^{\lambda_j,\lambda}}\right)_{\leq i}-\left(\sum_{j \leq k-1} n_{\lambda_j,\lambda_k} \overline{m^{\lambda_j,\lambda}}\right)_{\leq i}-\left(\sum_{j \leq k} n_{\lambda_j,\lambda_k} \overline{m^{\lambda_j,\lambda}}\right)_{\leq i-1}=(\overline{m^{\lambda_k,\lambda}})_i
\end{equation*}
which is what we want.
\end{proof}

\printbibliography
\end{document}